\documentclass[10pt]{amsart}

\usepackage{verbatim}
\usepackage{eucal,url,amssymb,stmaryrd,
enumerate,amscd,
}

\usepackage{amsfonts}
\usepackage{amsmath,amsthm,amssymb,amscd,enumerate,eucal,url,stmaryrd}

\numberwithin{equation}{section}

\newtheorem{thrm}{Theorem}[section]

\newtheorem{lemma}[thrm]{Lemma}

\newtheorem{prop}[thrm]{Proposition}

\newtheorem{cor}[thrm]{Corollary}

\newtheorem{dfn}[thrm]{Definition}

\newtheorem{rmrk}[thrm]{Remark}

\def\frg{{\frak g}}
\def\frh{{\frak h}}

\def\frn{{\frak n}}
\def\frs{{\frak s}}

\def\fru{{\frak u}}
\def\gc{\frg_\mathbb{C}}
\def\Re{{\frak R}{\frak e}\,}
\def\Im{{\frak I}{\frak m}\,}

\def\zzz{{\!\!\!}}

\begin{document}

\begin{abstract}
We classify non-nilpotent complex structures on 6-nilmanifolds and
their associated invariant balanced metrics. As an application we
find a large family of solutions of the heterotic supersymmetry
equations with non-zero flux, non-flat instanton and constant
dilaton satisfying the anomaly cancellation condition with respect
to the Chern connection.
\end{abstract}

\title[Non-nilpotent complex geometry of nilmanifolds and heterotic
supersymmetry] {Non-nilpotent complex geometry of nilmanifolds\\ and
heterotic supersymmetry
}

\author{Luis Ugarte}
\address[L. Ugarte]{Departamento de Matem\'aticas\,-\,I.U.M.A.,
Universidad de Zaragoza,
Campus Plaza San Francisco,
50009 Zaragoza, Spain} 
\email{ugarte@unizar.es}

\author{Raquel Villacampa}
\address[R. Villacampa]{Centro Universitario de la Defensa\,-\,I.U.M.A., Academia General
Mili\-tar, Crta. de Huesca s/n. 50090 Zaragoza, Spain}
\email{raquelvg@unizar.es}

\maketitle

\section{Introduction}

\noindent The complex geometry of nilmanifolds provides a rich
source of explicit examples of compact complex manifolds admitting
additional special structures with interesting properties. Many
authors have studied several aspects of this geometry from different
points of view (see for example~\cite{CFGU2,FIUV,FG,FPS,R,S,U} and
the references therein). Here we consider the case when the
invariant complex structure is not nilpotent, in the sense of
\cite{CFGU2}, and the nilmanifold has dimension six. This is the
lowest dimension when these structures appear. In
Section~\ref{complex-sect} we prove that there exist only 4
non-nilpotent complex structures $J_{\epsilon}^{\pm}$ $(\epsilon=0,1)$
up to equivalence. According to \cite{S,U} there are, up to
isomorphism, only two 6-nilmanifolds admitting non-nilpotent complex
structures, although only one of them admits special Hermitian
metrics. In fact, we show that if $\epsilon=0$ then the Lie algebra
underlying the nilmanifold is isomorphic to $(0,0,0,12,23,14-35)$,
whereas if $\epsilon=1$ then the corresponding Lie algebra is
$(0,0,12,13,23,14+25)$.

Along this paper $N$ denotes a nilmanifold with
$(0,0,0,12,23,14-35)$ as underlying Lie algebra. It turns out that
any invariant complex structure on the manifold $N$ is equivalent to
$J_0^+$ or $J_0^-$, and admits Hermitian structures which are
balanced in the sense of \cite{Mi}. Our goal in this paper is the
study of the balanced Hermitian geometry of $N$ and its application
to the heterotic supersymmetry theory with fluxes.

Strominger investigated in \cite{Str} the heterotic superstring
background with spacetime supersymmetry and proposed a model based
on Hermitian 6-manifolds with holomorphically trivial canonical
bundle which are Calabi-Yau with torsion, i.e. the holonomy of the
Bismut \cite{Bis} connection $\nabla^+$ is reduced to SU(3). More
concretely, as it is described in Section~\ref{het-constant-dil},
the Strominger system in heterotic string theory consists of the
gravitino, dilatino and gaugino equations, together with the anomaly
cancellation condition. Given an SU(3)-structure $(J,F,\Psi)$ on $M$,
where $F$ denotes the fundamental form,
the gravitino and dilatino equations are both satisfied if and
only if the holonomy of $\nabla^+$ reduces to SU(3), the structure
$J$ is integrable and the Lee form $\theta=2d\phi$, $\phi$ being the
dilaton function (see \cite{FI} for general necessary conditions to
solve the gravitino equation). The vanishing of the gaugino
variation requires a 2-form $\Omega^A\not=0$ of instanton type
\cite{Str}, i.e. a Donaldson-Uhlenbeck-Yau SU($3$)-instanton $A$
with non-zero curvature $\Omega^A$. Finally, the anomaly
cancellation condition in the Strominger system reads as
$$dT=\frac{\alpha'}{4}\left({\rm
tr}\, \Omega\wedge\Omega -{\rm tr}\, \Omega^A\wedge \Omega^A
\right),\quad\quad \alpha'>0,$$ where $T=JdF$ is the torsion of
$\nabla^+$. Here $\Omega$ is the curvature of some metric connection
and there are several proposals for it. In \cite{y4,y3,Str} the
curvature of the Chern connection $\nabla^c$ is proposed to
investigate the anomaly cancellation and, based on a construction in
\cite{GP}, Fu and Yau first proved the existence of solutions of the
Strominger system on a Hermitian non-K\"ahler manifold given as a
$\mathbb{T}^2$-bundle over a $K3$ surface \cite{y3}. In contrast,
torus bundles over a complex torus cannot solve the system with
respect to the Chern connection in the anomaly cancellation
condition~\cite{y4}.

The solutions given in \cite{y3} have non-constant dilaton function.
If one requires the dilaton to be constant then the Hermitian
structure must be balanced. An interesting fact is that for
nilmanifolds with invariant Hermitian structure the holonomy of
$\nabla^+$ reduces to SU(3) if and only if the structure is balanced
\cite{FPS}, that is to say, the gravitino equation is equivalent to
the dilatino equation with constant dilaton. The classification of
6-nilmanifolds admitting invariant balanced Hermitian structure is
given in \cite{U}, and it follows from \cite[Theorem~4.4]{R} that
the corresponding complex nilmanifold is always a torus bundle over
a complex 2-torus, except for the nilmanifold $N$. Therefore,
according to the result of \cite{y4} mentioned above, in order to
find a 6-nilmanifold solving the heterotic supersymmetry equations
with non-zero flux $T$, constant dilaton and with respect to
$\nabla^c$ in the anomaly cancellation, we are necessarily led to
investigate the non-nilpotent complex geometry of $N$ and their
associated invariant balanced metrics.

In Section~\ref{soluciones} we find many explicit solutions of the
Strominger system; actually, we prove that for any non-nilpotent
complex structure $J$ on $N$, the compact complex manifold $(N,J)$
admits a family of solutions of the heterotic supersymmetry
equations with non-zero flux, non-flat instanton $A$, i.e.
$\Omega^A\not=0$, and constant dilaton satisfying the anomaly
cancellation condition with respect to the Chern connection. Some of
these solutions are a deformation of the particular solution given
in \cite{FIUV} for which the instanton is flat (see
Remark~\ref{deformation} for details).

The paper is structured as follows. In Section~\ref{complex-sect} we
first prove that any non-nilpotent complex structure $J$ in six
dimensions is equivalent to $J_{\epsilon}^{\pm}$ $(\epsilon=0,1)$, and
it has balanced Hermitian metrics if and only if $\epsilon=0$, i.e.
$J$ lives on the nilmanifold $N$. Moreover, in
Theorem~\ref{families-balanced-h19} we describe the space of
invariant balanced $J$-Hermitian structures on $N$: any such a
structure is isomorphic to $(J_0^{\pm},F)$, where $F$ belongs to one of
two families to which we refer as Family I and Family II. After
recalling the heterotic supersymmetry system with constant dilaton,
we find in Section~\ref{het-constant-dil} an adapted basis for each
balanced Hermitian structure in the Families I and II. These adapted
bases are used in Section~\ref{soluciones} to explicitly construct a
3-parametric family of instantons $A_{\lambda,\mu,\tau}$ for Family
I such that ${\rm tr} (\Omega^{A_{\lambda,\mu,\tau}}\wedge
\Omega^{A_{\lambda,\mu,\tau}})\not=0$ if and only if the parameter
$\tau$ is non-zero. In Theorem~\ref{solutions-Family-I} we find many
solutions in Family I of the heterotic supersymmetry equations with
non-zero flux, non-flat instanton and constant dilaton with respect
to the Chern connection in the anomaly cancellation condition. For
Family II we also get in Theorem~\ref{solutions-Family-II} many
solutions to the supersymmetry system, although in this case the
instantons we find become flat. Finally, in Section~\ref{explicit}
we show an explicit realization of the nilmanifold $N$ and of the
balanced Hermitian structures involved.

\section{Non-nilpotent complex structures on 6-nilmanifolds and compatible balanced
metrics}\label{complex-sect}

\noindent Let $M$ be a nilmanifold, i.e. a compact quotient of a
simply-connected nilpotent Lie group $G$ by a lattice $\Gamma$ of
maximal rank. Any left-invariant complex structure on $G$ descends
to $M$ in a natural way, so a source (possibly empty) of complex
structures on $M$ is given by the endomorphisms $J\colon\frg
\longrightarrow \frg$ of the Lie algebra $\frg$ of $G$ such that
$J^2=-{\rm Id}$ and satisfying the ``Nijenhuis condition''
$$
[JX,JY]=J[JX,Y]+J[X,JY]+[X,Y],
$$
for any $X,Y\in \frg$. We shall refer to any such an endomorphism as
a {\it complex structure} on the Lie algebra~$\frg$.

Associated to a complex structure $J$, there exists an ascending
series $\{\frg_l^J\}_{l\geq 0}$ on the Lie algebra, defined
inductively by
$$
\frg_0^J = \{0\} \ , \qquad  \frg_l^J = \{ X \in \frg \, \mid \, [X,
\frg] \subseteq \frg_{l-1}^J \ \ \mbox{\rm and} \ \ [JX, \frg]
\subseteq \frg_{l-1}^J \}\ , \quad l\geq 1.
$$
For any $l\geq 0$, the term $\frg_l^J$ is a $J$-invariant ideal of
$\frg$ which is contained in the term $\frg_l=\{X\in\frg\mid
[X,\frg]\subseteq \frg_{l-1}\}$ of the usual ascending central
series of $\frg$. But whereas $\frg_l$ always reach the whole Lie
algebra $\frg$ when $\frg$ is nilpotent, the series
$\{\frg_l^J\}_{l\geq 0}$ can stabilize in a proper $J$-ideal of
$\frg$. This motivates the following terminology: If $\frg_l^J=\frg$
for some $l$, then the complex structure $J$ is called {\it
nilpotent}~\cite{CFGU2}; otherwise we shall say that $J$ is {\it
non-nilpotent}.

\begin{dfn}\label{def-non-nilp}
{\rm A {\it nilpotent} (resp. {\it non-nilpotent}) complex structure
on a nilmanifold $M$ is a complex structure on $M$ coming from a
nilpotent (resp. non-nilpotent) complex structure $J$ on the
underlying Lie algebra $\frg$.}
\end{dfn}

Let us denote by $\gc$ the complexification of $\frg$ and by $\gc^*$
its dual. Given an endomorphism $J\colon \frg \longrightarrow \frg$
such that $J^2=-{\rm Id}$, we denote by $\frg^{1,0}$ and
$\frg^{0,1}$ the eigenspaces of the eigenvalues $\pm i$ of $J$ as an
endomorphism of~$\gc^*$, respectively. Then, the decomposition
$\gc^*=\frg^{1,0}\oplus\frg^{0,1}$ induces a natural bigraduation on
the complexified exterior algebra $\bigwedge^* \,\gc^* =\oplus_{p,q}
\bigwedge^{p,q}(\frg^*)=\oplus_{p,q} \bigwedge^p(\frg^{1,0})\otimes
\bigwedge^q(\frg^{0,1})$. If $d$ denotes the usual
Chevalley-Eilenberg differential of the Lie algebra, we shall also
denote by $d$ its extension to the complexified exterior algebra,
i.e. $d\colon \bigwedge^* \gc^* \longrightarrow \bigwedge^{*+1}
\gc^*$. It is well-known that the endomorphism $J$ is a complex
structure if and only if $d(\frg^{1,0})\subset
\bigwedge^{2,0}(\frg^*)\oplus \bigwedge^{1,1}(\frg^*)$. In the case
of nilpotent Lie algebras $\frg$, Salamon proves in~\cite{S} the
following equivalent condition for the endomorphism $J$ to be a
complex structure: $J$ is a complex structure on $\frg$ if and only
if $\frg^{1,0}$ has a basis $\{\omega^j\}_{j=1}^n$ such that
$d\omega^1=0$ and
$$
d \omega^{j} \in \mathcal{I} (\omega^1,\ldots,\omega^{j-1}), \quad
\mbox{ for } j=2,\ldots,n ,
$$
where $\mathcal{I} (\omega^1,\ldots,\omega^{j-1})$ is the ideal in
$\bigwedge\phantom{\!}^* \,\gc^*$ generated by
$\{\omega^1,\ldots,\omega^{j-1}\}$. From now on, we shall denote
$\omega^{j}\wedge\omega^k$ and $\omega^{j}\wedge\overline{\omega^k}$
simply by $\omega^{jk}$ and $\omega^{j\bar{k}}$, respectively.

In dimension six, the non-nilpotent complex structures are
characterized as follows:

\begin{prop}\label{J-red}\cite[Proposition~2~(a)]{U}
Let $J$ be a non-nilpotent complex structure on a nilpotent Lie
algebra $\frg$ of dimension 6. Then, there is a basis
$\{\omega^j\}_{j=1}^3$ for $\frg^{1,0}$ such that
\begin{equation}\label{nonnilpotent}
\left\{
\begin{array}{lcl}
d\omega^1 \zzz & = &\zzz 0,\\
d\omega^2 \zzz & = &\zzz E\, \omega^{13} + \omega^{1\bar{3}} \, ,\\
d\omega^3 \zzz & = &\zzz A\, \omega^{1\bar{1}} + ib\,
\omega^{1\bar{2}} - ib\bar{E}\, \omega^{2\bar{1}}  ,
\end{array}
\right.
\end{equation}
where $A,E\in \mathbb{C}$ with $|E|=1$, and $b\in \mathbb{R}-\{0\}$.
\end{prop}

A complex structure $J$ on $\frg$ is said to be {\it equivalent} to
a complex structure $J'$ on $\frg'$ if there is an isomorphism
$A\colon \frg \longrightarrow \frg'$ of Lie algebras such that
$J'A=AJ$. Equivalently, the linear isomorphism $A^*\colon \frg'^*
\longrightarrow \frg^*$ commutes with the Chevalley-Eilenberg
differentials and the extension of $A^*$ to the complexified
exterior algebra preserves the bigraduations induced by $J$ and
$J'$.

In the next result we prove that up to equivalence there are only
four non-nilpotent complex structures in dimension six.

\begin{prop}\label{str-eq-red-h19}
Let $J$ be a non-nilpotent complex structure on a 6-dimensional
nilpotent Lie algebra $\frg$. Then, there is a $(1,0)$-basis
$\{\omega^j\}_{j=1}^3$ satisfying
\begin{equation}\label{non-nilp-reducido} \begin{cases}
\begin{array}{lcl}
d\omega^1 \zzz & = &\zzz 0,\\
d\omega^2 \zzz & = &\zzz  \omega^{13} + \omega^{1\bar{3}} \, ,\\
d\omega^3 \zzz & = &\zzz  i\epsilon\, \omega^{1\bar{1}} \pm i\,
(\omega^{1\bar{2}} - \omega^{2\bar{1}}) ,
\end{array}
\end{cases}
\end{equation}
with $\epsilon=0,1$. The four complex structures $J_{\epsilon}^{\pm}$
$(\epsilon=0,1)$ corresponding to the different choices of the
coefficients in $(\ref{non-nilp-reducido})$ are not equivalent.
\end{prop}

\begin{proof}
From Proposition~\ref{J-red} we have the existence of a (1,0)-basis
$\{\omega''^j\}_{j=1}^3$ satisfying
$$
d\omega''^1 = 0,\quad d\omega''^2 = E\, \omega''^{13} +
\omega''^{1\bar{3}},\quad d\omega''^3 = A\, \omega''^{1\bar{1}} +
ib\, \omega''^{1\bar{2}} - ib\bar{E}\, \omega''^{2\bar{1}},
$$
where $A,E\in \mathbb{C}$ with $|E|=1$ and $b\in \mathbb{R}-\{0\}$.
The coefficients of these equations are easily seen to be reduced to
$E=1$ and $b=\pm 1$ when we consider the new (1,0)-basis $\{
\omega'^1=\sqrt{|b|}\, \omega''^1,\ \omega'^2=\bar{\vartheta}
\sqrt{|b|}\, \omega''^2,\ \omega'^3=\vartheta\, \omega''^3\}$,
$\vartheta$ being a non-zero solution of $\bar{\vartheta}
E=\vartheta$. In fact, in terms of the basis $\{\omega'^j\}_{j=1}^3$
the structure equations become
$$d\omega'^1=0,\quad d\omega'^2=\omega'^{13} +
\omega'^{1\bar3},\quad d\omega'^3=B\,\omega'^{1\bar1}\pm
i\,(\omega'^{1\bar2}-\omega'^{2\bar1}),$$ for some $B\in
\mathbb{C}$. Let us denote $x=\Re B$ and $y=\Im B$. Now, if
$y\not=0$ then with respect to the new (1,0)-basis $\{ \omega^1= i\,
\omega'^1,\ \omega^2= \mp {x\over 2y}\, \omega'^1 + {i\over y}\,
\omega'^2,\ \omega^3= {1\over y}\, \omega'^3 \}$ the equations above
reduce to (\ref{non-nilp-reducido}) with $\epsilon=1$. In other
case, if $B=x\in \mathbb{R}-\{0\}$ then we get
(\ref{non-nilp-reducido}) with $\epsilon=0$ with respect to $\{
\omega^1= i\, \omega'^1, \omega^2= \pm \omega'^1 - {2i\over x}\,
\omega'^2, \omega^3= -{2\over x}\, \omega'^3 \}$.

Finally, let us consider $\{\omega^j\}_{j=1}^3$ and
$\{\sigma^j\}_{j=1}^3$ $(1,0)$-bases corresponding to two complex
structures given by \eqref{non-nilp-reducido}, i.e.
\begin{equation*}
\begin{array}{l}
d\omega^1=0,\quad d\omega^2=\omega^{13}+\omega^{1\bar3},\quad
d\omega^3= i\epsilon\, \omega^{1\bar{1}} + i c\,(\omega^{1\bar2} - \omega^{2\bar1}),\\[8pt]
d\sigma^1=0,\quad d\sigma^2=\sigma^{13}+\sigma^{1\bar3},\quad
d\sigma^3= i\epsilon'\, \sigma^{1\bar{1}} + i c'\,(\sigma^{1\bar2} -
\sigma^{2\bar1}),\end{array}
\end{equation*}
where $\epsilon,\epsilon'\in\{0,1\}$ and $c,c'\in\{\pm 1\}$. If we
express $\sigma^i=m_{i1}\,\omega^1 + m_{i2}\,\omega^2 +
m_{i3}\,\omega^3$ for $i= 1, 2, 3$ and $(m_{ij})\in$
GL($3,\mathbb{C}$), then $d\sigma^i=m_{i1}\,d\omega^1 +
m_{i2}\,d\omega^2 + m_{i3}\,d\omega^3$ is equivalent to
$$m_{12}=m_{13}=m_{23}=m_{31}=m_{32}=0,$$ $$m_{33}\in\mathbb{R} - \{0\},\quad \Im(m_{11}\overline
{m_{21}})=0,\quad m_{33}\, c= m_{11}\overline {m_{22}}\, c',\quad
m_{11}m_{33}=m_{22},\quad m_{33}\, \epsilon= |m_{11}|^2\,
\epsilon'.$$ These last conditions imply that $\epsilon=\epsilon'$
and $c=|m_{11}|^2\, c'$, so necessarily $c=c'$. In conclusion, there
does not exist any equivalence between any two of the complex
structures~\eqref{non-nilp-reducido}.
\end{proof}

The classification of 6-dimensional nilpotent Lie algebras admitting
non-nilpotent complex structure is as follows:

\begin{prop}\label{h19-h26}
If $\epsilon=0$ in $\eqref{non-nilp-reducido}$ then the Lie algebra
$\frg$ is isomorphic to $(0,0,0,12,23,14-35)$, whereas if
$\epsilon=1$ then $\frg\cong (0,0,12,13,23,14+25)$.
\end{prop}

\begin{proof}
For the case $\epsilon=0$ it suffices to consider the real basis
$\{\alpha^1,\ldots,\alpha^6 \}$ for $\frg^*$ given by $\alpha^1 -
i\, \alpha^3 = \omega^1$, $\alpha^4 + i\, \alpha^5 = \omega^2$ and
${1\over 2}\, \alpha^2 \pm 2i\, \alpha^6 = \omega^3$. For
$\epsilon=1$ it suffices to consider the basis
$\{\alpha^1,\ldots,\alpha^6 \}$ for $\frg^*$ given by ${\sqrt{2}
\over 2}(\alpha^1 + i\, \alpha^2) = \omega^1$, $\sqrt{2}(\alpha^4 +
i\, \alpha^5) = \omega^2$ and $\alpha^3 \pm 2i\, \alpha^6 =
\omega^3$.
\end{proof}

In the next result we summarize the previous propositions.

\begin{cor}\label{Jotas}
The Lie algebras $(0,0,0,12,23,14-35)$ and $(0,0,12,13,23,14+25)$
are the only 6-dimensional nilpotent Lie algebras admitting
non-nilpotent complex structure. Moreover, up to equivalence, each
of these Lie algebras has only two complex structures.
\end{cor}

\begin{rmrk}\label{real-clasification}
{\rm If a 6-dimensional nilpotent Lie algebra $\frg$ has a nilpotent
complex structure, then all the complex structures on $\frg$ must be
nilpotent. As it is proved in \cite[Corollary~7]{U}, this assertion
follows from the fact that the first term $\frg_l^J$ in the series
adapted to $J$, which is a $J$-invariant ideal of $\frg$ contained
in the center, is non-trivial; therefore, such a Lie algebra $\frg$
is not isomorphic neither to $(0,0,0,12,23,14-35)$ nor to
$(0,0,12,13,23,14+25)$ because they have 1-dimensional center. This
result, together with the classification of 6-dimensional nilpotent
Lie algebras admitting nilpotent complex structure obtained
in~\cite{CFGU1}, provides the classification of 6-dimensional
nilpotent Lie algebras admitting complex structure, given by Salamon
in~\cite{S}.}
\end{rmrk}

\begin{rmrk}\label{notation}
{\rm As it is observed in \cite{S}, the Lie algebras
$(0,0,0,12,23,14-35)$ and $(0,0,12,13,23,14+25)$ correspond to
$\frh_{19}^-$ and $\frh_{26}^+$, respectively, in the notation given
in \cite{CFGU1}.}
\end{rmrk}

Now, let $\frg$ be a Lie algebra of dimension 6. A {\it Hermitian
structure} on $\frg$ is a pair $(J,g)$, where $J$ is a complex
structure on $\frg$ and $g$ is an inner product on $\frg$ compatible
with $J$ in the usual sense, i.e. $g(\cdot,\cdot)=g(J\cdot,J\cdot)$.
The associated {\it fundamental form} $F\in \bigwedge^2 \frg^*$ is
defined by $F(X,Y)=g(X,JY)$ and expresses in terms of any basis
$\{\omega^j\}_{j=1}^3$, of type (1,0) with respect to $J$, by
\begin{equation}\label{2forma}
2\, F=i (r^2\omega^{1\bar{1}} + s^2\omega^{2\bar{2}} +
t^2\omega^{3\bar{3}})
+u\,\omega^{1\bar{2}}-\bar{u}\,\omega^{2\bar{1}} +
v\,\omega^{2\bar{3}}-\bar{v}\,\omega^{3\bar{2}}+
z\,\omega^{1\bar{3}}-\bar{z}\,\omega^{3\bar{1}},
\end{equation}
for some $r,s,t\in\mathbb{R}$ and $u,v,z\in\mathbb{C}$. Since we are
using the convention $(J\alpha)(X)=-\alpha(JX)$ for
$\alpha\in\frg^*$, the inner product $g$ is given by
\begin{equation}\label{metric}
g= r^2\,\omega^1\omega^{\bar{1}} + s^2\, \omega^2\omega^{\bar{2}} +
t^2\, \omega^3\omega^{\bar{3}} - \frac{i}{2}(u\,
\omega^1\omega^{\bar{2}} - \bar{u}\, \omega^2\omega^{\bar{1}} +v\,
\omega^2\omega^{\bar{3}} - \bar{v}\, \omega^3\omega^{\bar{2}}+ z\,
\omega^1\omega^{\bar{3}} -\bar{z}\, \omega^3\omega^{\bar{1}}).
\end{equation}
Here $\omega^j\omega^{\bar{k}}=\frac12
(\omega^j\otimes\omega^{\bar{k}} + \omega^{\bar{k}}\otimes\omega^j)$
denotes the symmetric product of $\omega^j$ and $\omega^{\bar{k}}$.
Notice that the positive definiteness of $g$ implies that the
coefficients $r^2,\,s^2,\,t^2$ are non-zero real numbers and
$u,\,v,\,z\in \mathbb{C}$ satisfy $r^2s^2>|u|^2,$
$s^2t^2>|v|^2,\,r^2t^2>|z|^2$ and $r^2s^2t^2 + 2\,\Re(i\bar u\bar
vz)>t^2|u|^2 + r^2|v|^2 + s^2|z|^2$.

Fixed $J$, since $g$ and $F$ are mutually determined by each other,
we shall also denote the Hermitian structure $(J,g)$ by the pair
$(J,F)$. Recall that the Hermitian structure $(J,F)$ is said to be
{\it balanced} if $F^2$ is a closed form or, equivalently, $F\wedge
dF=0$. Our goal in this section is to classify the balanced
Hermitian structures $(J,F)$ on 6-dimensional nilpotent Lie algebras
$\frg$ with $J$ non-nilpotent, up to equivalence. We recall that a
Hermitian structure $(J,F)$ on $\frg$ is said to be {\it equivalent}
to a Hermitian structure $(J',F')$ on $\frg'$ if there is an
isomorphism $A\colon \frg \longrightarrow \frg'$ of Lie algebras
such that $J'A=AJ$ and $F=A^*F'$. Notice that this implies that
$g=A^*g'$, where $g$ and $g'$ denote the inner products associated
to $(J,F)$ and $(J',F')$, respectively.

The only 6-dimensional nilpotent Lie algebra admitting balanced
Hermitian metric with respect to a non-nilpotent complex structure
is $(0,0,0,12,23,14-35)$, which from now on we shall denote it by
$\frn$ instead of $\frh_{19}^-$. Next we describe the balanced
Hermitian metrics on $\frn$ with respect to a (1,0)-basis satisfying
\eqref{non-nilp-reducido} for $\epsilon=0$.

\begin{prop}\label{unique-balanced}
If $\frg$ is a 6-dimensional nilpotent Lie algebra admitting
balanced Hermitian structure $(J,F)$ with respect to a non-nilpotent
complex structure $J$, then $\frg\cong \frn=(0,0,0,12,23,14-35)$.
Moreover, there is a $(1,0)$-basis $\{\omega^j\}_{j=1}^3$ satisfying
\begin{equation}\label{h19complexreducido}
d\omega^1 = 0,\quad\ d\omega^2 =  \omega^{13} +
\omega^{1\bar{3}},\quad\ d\omega^3 = \pm i\, (\omega^{1\bar{2}} -
\omega^{2\bar{1}}),
\end{equation}
such that the fundamental form $F$ is given by
\begin{equation}\label{h19fundform}
2F=i(r^2\,\omega^{1\bar1} + s^2\,\omega^{2\bar2} +
t^2\,\omega^{3\bar3}) + u\,\omega^{1\bar2} -\bar u\,\omega^{2\bar1}
+ v\,\omega^{2\bar3} -\bar v\,\omega^{3\bar2}+ z\,\omega^{1\bar3}
-\bar z\,\omega^{3\bar1},
\end{equation}
with
\begin{equation}\label{balanced_condition_h19}
u+\bar u=0, \quad\quad z=-i uv/s^2.
\end{equation}
\end{prop}

\begin{proof}
From \cite[Proposition~25~(a)]{U} we have that if $J$ is a
non-nilpotent complex structure defined
by~$(\ref{non-nilp-reducido})$, then $(J,F)$ is balanced if and only
if $i \epsilon s^2 \pm(u+\bar{u})=0$ and $z=-iuv/s^2$, which is
equivalent to the conditions
$$
\epsilon=0,\quad\quad u+\bar{u}=0 \quad\quad \mbox{and} \quad\quad
z=-iuv/s^2.
$$
The result follows from Proposition~\ref{h19-h26}.
\end{proof}

Notice that as a consequence of this result any complex structure
$J$ on $\frn$ admits balanced $J$-Hermitian structures. Next we
classify balanced Hermitian structures on the Lie algebra $\frn$ up
to equivalence. First, we need the following result which describes
the automorphisms of the Lie algebra $\frn$ preserving $J_0^+$ or
$J_0^-$.

\begin{lemma}\label{cambioestructurah19}
Let $\{\omega^j\}^3_{j=1}$ be a $(1,0)$-basis satisfying
$\eqref{h19complexreducido}$. Then, the transformations given by
$\{\omega'^{1}=e^{i\theta}\,\omega^1,\,
\omega'^{2}=B\,\omega^1+c\,e^{i\theta}\omega^2,\,
\omega'^{3}=c\,\omega^3\}$, where $c\in \mathbb{R}-\{0\}$,
$\theta\in [0,2\pi),\,B\in\mathbb C$ such that
$\Im(e^{i\theta}\,\bar B)=0$, are the only ones preserving the
structure equations $\eqref{h19complexreducido}$.
\end{lemma}

\begin{proof}
It is immediate to check that any basis $\{\omega'^{j}\}_{j=1}^3$
satisfies the equations \eqref{h19complexreducido} if
$\{\omega^j\}_{j=1}^3$ does. On the other hand, consider
$\{\omega^j\}_{j=1}^3$ and $\{\omega'^{j}\}_{j=1}^3$
satisfying~\eqref{h19complexreducido} for the same choice of sign.
If we write $\omega'^{i}=m_{i1}\,\omega^1 + m_{i2}\,\omega^2 +
m_{i3}\,\omega^3$, for $i=1,2,3$, with $(m_{ij})\in {\rm
GL}(3,\mathbb C)$, then it is easy to see that
$d\omega'^{i}=m_{i1}\,d\omega^1 + m_{i2}\,d\omega^2 +
m_{i3}\,d\omega^3$ is equivalent to
$$m_{12}=m_{13}=m_{23}=m_{31}=m_{32}=0,\quad \Im (m_{11}\overline{m_{21}})=0,$$
$$m_{33}\in \mathbb{R}-\{0\},\quad m_{11}m_{33}=m_{22}
,\quad m_{11}\overline{m_{22}}=m_{33}.$$  The last two conditions
imply $|m_{11}|^2=1$, so $m_{11}=e^{i\theta}$. If we put $m_{33}=c$,
then $m_{22}=c\,e^{i\theta}$ and finally,
$\Im(e^{i\theta}\,\overline{m_{21}})=0$.
\end{proof}

\begin{thrm}\label{families-balanced-h19}
Let $(J,F)$ be a balanced Hermitian structure on $\frn$ for
$J=J_0^{\pm}$. Then, in the conditions of {\rm
Proposition~\ref{unique-balanced}}, the fundamental form $F$ is
equivalent to one and only one form in the following families:

\medskip

{\rm Family I:} \quad
$2F=i(r^2\,\omega^{1\bar1}+s^2\,\omega^{2\bar2}+\omega^{3\bar3}),\quad
r,s\not=0$;

\medskip

{\rm Family II:} \ \
$2F=i(r^2\,\omega^{1\bar1}+s^2\,\omega^{2\bar2}+t^2\,
\omega^{3\bar3})+\omega^{2\bar3}-\omega^{3\bar2},\quad r\not= 0,
\quad s^2t^2>1$.
\end{thrm}

\begin{proof}
Let $\{\sigma^j\}^3_{j=1}$ be a $(1,0)$-basis
satisfying~\eqref{h19complexreducido} and let us consider a general
compatible metric~\eqref{metric} satisfying the balanced
conditions~\eqref{balanced_condition_h19}. Let us write $v$ as
$v=-c^2\,e^{i\theta}=c^2\,e^{i(\theta-\pi)}$. If $v=0$ then taking
$\{\omega^1=-\sigma^1,\,\omega^2=
t\,\left(-\frac{ui}{s^2}\,\sigma^1+\sigma^2\right),\,\omega^3=-t\,\sigma^3\}$
we obtain a balanced structure in Family II. On the other hand, if
$v\neq 0$ then the basis
$\{\omega^1=-e^{i\theta}\,\sigma^1,\,\omega^2=ce^{i\theta}(-\frac{ui}{s^2}\,\sigma^1+\sigma^2),\,
\omega^3=-c\,\sigma^3\}$ keeps~\eqref{h19complexreducido} by
Lemma~\ref{cambioestructurah19} (observe that
$\Im\left(\frac{iu}{s^2}\right)=0$ because $u\in i \mathbb{R}$) and
provides a balanced structure in Family I, concretely
$$2F=i\left(\frac{r^2s^2-|u|^2}{s^2}\,\omega^{1\bar1} + \frac{s^2}{c^2}\,\omega^{2\bar2}
+ \frac{t^2}{c^2}\,\omega^{3\bar3}\right) +
\omega^{2\bar3}-\omega^{3\bar2}.$$

Finally, we see that the structures in Families I and II are not
equivalent. Consider the pairs
$$\{\omega^1,\,\omega^2,\,\omega^3\},\quad
F_{\omega}=i\,(r^2\,\omega^{1\bar1}+s^2\,\omega^{2\bar2}+t^2\,\omega^{3\bar3})
+ \varepsilon_1\,(\omega^{2\bar3}-\omega^{3\bar2}),$$ and
$$\{\sigma^1,\,\sigma^2,\,\sigma^3\},\quad
F_{\sigma}=i\,(a^2\,\sigma^{1\bar1}+b^2\,\sigma^{2\bar2}+c^2\,\sigma^{3\bar3})
+ \varepsilon_2\,(\sigma^{2\bar3}-\sigma^{3\bar2}),$$ where
$\{\omega^j\}_{j=1}^3$ and $\{\sigma^j\}_{j=1}^3$ are $(1,0)$-basis
satisfying \eqref{h19complexreducido} for the same complex
structure. Let us suppose that there exists an equivalence between
them. By Lemma~\ref{cambioestructurah19},
$\omega^1=e^{i\theta}\,\sigma^1$, $\omega^2=B\sigma^1 + \lambda
e^{i\theta}\,\sigma^2$ and $\omega^3=\lambda\,\sigma^3$ for some
$\lambda\in \mathbb{R}-\{0\},\,\theta\in [0,2\pi),\,B\in \mathbb{C}$
such that $\Im(e^{i\theta} \bar B)=0$. Then, the conditions that
transform $F_{\omega}$ into $F_{\sigma}$ are:
$$i(\lambda\,\bar B\,e^{i\theta})=\varepsilon_1\,B\lambda=0,\quad\
\varepsilon_1\,\lambda^2\,e^{i\theta}=\varepsilon_2,\quad\
r^2+s^2|B|^2 = a^2,\quad\  s^2\lambda^2=b^2,\quad\
t^2\lambda^2=c^2.$$ From the first relation we have that $B=0$ and,
in consequence, $r^2=a^2$. Now, if $\varepsilon_1=\varepsilon_2=1$,
then $\lambda^2=e^{i\theta}=1$ and therefore, $s^2=b^2$ and
$t^2=c^2$. If $\varepsilon_1=\varepsilon_2=0$ and $t^2=c^2=1$, then
$\lambda^2=1$ and again $s^2=b^2$. Finally, if $\varepsilon_1=1$ and
$\varepsilon_2=0,\, c^2=1$, then $\lambda^2 e^{i\theta}=0$ which is
not possible.
\end{proof}

\section{Heterotic supersymmetry with constant
dilaton}\label{het-constant-dil}

\noindent Conformally balanced Hermitian structures are a key
ingredient in the solutions of the Strominger system of equations in
heterotic string theory~\cite{Str}. Before writing the system
explicitly, we first recall that any Hermitian structure $(J,F)$ on
a $2n$-dimensional manifold $M$ has a unique Hermitian connection
with torsion $T$ given by $g(X,T(Y,Z))=JdF(X,Y,Z)=-dF(JX,JY,JZ)$,
$g$ being the associated metric~\cite{Bis}. This torsion connection
is known as the {\it Bismut connection} of $(J,F)$ and will be
denoted here by $\nabla^+$. From now on, we shall identify $T$ with
the 3-form $JdF$. In relation to the Levi-Civita connection
$\nabla^g$ of the Riemannian metric $g$, the Bismut connection is
determined by $\nabla^+ = \nabla^g +\frac{1}{2}T$. On the other
hand, let $\theta$ be the {\it Lee form} associated to the Hermitian
structure $(J,F)$, that is, $\theta=\frac{1}{1-n} J\delta F$, where
$\delta$ denotes the formal adjoint of $d$ with respect to the
associated metric $g$. The 1-form $\theta$ vanishes identically if
and only if the form $F^{n-1}$ is closed, i.e. the Hermitian
structure is balanced.

In addition to a conformally balanced Hermitian structure, the
Strominger system also requires $M$ to be a manifold of complex
dimension 3 endowed with a non-vanishing holomorphic (3,0)-form
$\Psi$. Therefore, one is led to consider certain special Hermitian
SU(3)-structures $(J,F,\Psi)$ in six dimensions \cite{CS}.

Here we are interested in finding explicit compact solutions of the
heterotic supersymmetry equations with non-zero flux $H=T$. More
concretely, in the heterotic theory of superstrings with fluxes one
looks for a compact 6-manifold $M$ endowed with a Hermitian
SU(3)-structure $(J,F,\Psi)$ satisfying the following system of
equations \cite{Str}:
\begin{enumerate}
\item[(a)] Gravitino equation: the holonomy of the Bismut connection $\nabla^+$ is contained in
SU(3).
\item[(b)] Dilatino equation: the Lee form $\theta$ is exact; moreover $\theta=2d\phi$, where
$\phi$ is the dilaton function.
\item[(c)] Gaugino equation: there is a Donaldson-Uhlenbeck-Yau SU($3$)-instanton; that is to
say, a Hermitian vector bundle of rank $r$ over $M$ equiped with an
SU($3$)-instanton, i.e. a connection $A$ with curvature 2-form
$\Omega^A\in \frs\fru(3)$.
\item[(d)] Anomaly cancellation condition:
$$dT=2\pi^2 \alpha' \Big(p_1(\nabla)-p_1(A)\Big), \qquad
\alpha'>0.$$
\end{enumerate}

Here $p_1$ denotes the 4-form representing the first Pontrjagin
class of the connection, which is given in terms of the curvature
forms $\Omega^i_j$ of the connection by $p_1= {1\over 8\pi^2} {\rm
tr}\, \Omega\wedge\Omega = {1\over 8\pi^2} \sum_{1\leq i<j\leq 6}
\Omega^i_j\wedge\Omega^i_j$. The requirement $\alpha'>0$ in the
anomaly cancellation condition is due to physical considerations.
Different connections $\nabla$ in (d) correspond to different
regularization schemes in the two-dimensional worldsheet non-linear
sigma model. In \cite{y4,y3,Str} the Chern connection $\nabla^c$ is
proposed to investigate the anomaly cancellation.

As an application of the results given in
Section~\ref{complex-sect}, next we find solutions to the system
above in the case when the dilaton is constant, which implies that
the structure $(J,g)$ is balanced Hermitian, and satisfying the
anomaly cancellation condition with curvature term taken with
respect to the Chern connection, i.e.~$\nabla=\nabla^c$.

We focus on six-dimensional nilmanifolds $M=\Gamma\backslash G$
endowed with an invariant non-nilpotent complex structure $J$. As it
is explained in the introduction, this is the appropriate class of
structures to look for solutions of the Strominger system with
respect the Chern connection in the anomaly cancellation condition.
According to \cite[Proposition~6.1]{FPS}, for invariant Hermitian
metrics on $M$ the balanced condition is equivalent to the reduction
of the holonomy group of the associated Bismut connection $\nabla^+$
to a subgroup of SU(3), i.e. the conditions (a) and (b) in the
system above are equivalent. Also, notice that given an invariant
Hermitian structure on a nilmanifold any invariant (3,0)-form $\Psi$
is closed \cite{S}.

Given a Hermitian structure $(J,F)$, we say that a (real) coframe
$\{e^1,\ldots,e^6\}$ is an {\it adapted basis} for~$(J,F)$ if
$Je^1=-e^2, Je^3=-e^4, Je^5=-e^6$ and $F=e^{12}+e^{34}+e^{56}$. The
associated metric $g$ expresses in terms of this basis by
$g=e^1\otimes e^1+\cdots+ e^6\otimes e^6$, i.e. $\{e^1,\ldots,e^6\}$
is orthonormal. In the subsequent sections we shall always consider
an adapted basis and the (3,0)-form $\Psi$ defining the
SU(3)-structure as
$\Psi=(e^1+ie^2)\wedge(e^3+ie^4)\wedge(e^5+ie^6)$.

There is a clear advantage when working with adapted bases. For
instance, in the gaugino equation the curvature 2-form $\Omega^A\in
{\frak s}{\frak u}(3)$ if and only if
\begin{equation}\label{inst}
(\Omega^A)^i_j(e_1,e_2)+(\Omega^A)^i_j(e_3,e_4)+(\Omega^A)^i_j(e_5,e_6)=0,\quad
(\Omega^A)^i_j(Je_k,Je_l)=(\Omega^A)^i_j(e_k,e_l),\ \ \forall\
i,j,k,l,
\end{equation}
where $\{e_1,\ldots,e_6\}$ is the dual basis of a basis
$\{e^1,\ldots,e^6\}$ adapted to the SU(3)-structure. Moreover, since
we are working with invariant structures, the adapted bases are
globally defined on the nilmanifold $M$ because they stem from the
corresponding Lie group $G$ (actually, from the Lie algebra $\frg$
of $G$) by passing to the quotient.

In general, for any 6-dimensional Lie group, let us consider the
structure equations with respect to a basis of left-invariant
1-forms:
$$
d\, e^k = \sum_{1\leq i<j \leq 6} a_{ij}^k \, e^{i j},\quad\quad
k=1,\ldots,6.
$$
Let $g=e^1\otimes e^1 + \cdots +e^6\otimes e^6$ be the Riemannian
metric for which the basis $\{e^k\}^6_{k=1}$ is orthonormal, and
denote by $\{e_1,\ldots,e_6\}$ the dual basis.

Given a linear connection $\nabla$, the connection 1-forms
$\sigma^i_j$ with respect to the above basis are
$$
\sigma^i_j(e_k) = g(\nabla_{e_k}e_j,e_i),
$$
i.e. $\nabla_X e_j = \sigma^1_j(X)\, e_1 +\cdots+ \sigma^6_j(X)\,
e_6$. The curvature 2-forms $\Omega^i_j$ of $\nabla$ are given in
terms of the connection 1-forms $\sigma^i_j$ by
\begin{equation}\label{curvature}
\Omega^i_j  = d \sigma^i_j + \sum_{1\leq k \leq 6}
\sigma^i_k\wedge\sigma^k_j
\end{equation}
and the first Pontrjagin class is represented by the 4-form
$p_1(\nabla)={1\over 8\pi^2} \sum_{1\leq i<j\leq 6}
\Omega^i_j\wedge\Omega^i_j$.

Now, if $J$ is a complex structure compatible with $g$ then, in
addition to the Bismut connection $\nabla^+=\nabla^g + \frac12 T$,
we can also consider the Chern connection $\nabla^c$, which is
defined by $\nabla^c= \nabla^g + \frac12 C$ with
$C(\cdot,\cdot,\cdot)=dF(J\cdot,\cdot,\cdot)$, that is
$$
g(\nabla^c_XY,Z)= g(\nabla^g_XY,Z) + \frac12 C(X,Y,Z), \qquad
C(X,Y,Z)=dF(JX,Y,Z)=-T(X,JY,JZ).
$$
Observe that $C(X,\cdot,\cdot)=(JX\lrcorner dF)(\cdot,\cdot)$ is a
2-form and $C(\cdot,Y,Z)=(J\cdot\lrcorner dF)(Y,Z)$ a 1-form.

Since $d e^k(e_i,e_j)= -e^k([e_i,e_j])$, the Levi-Civita connection
1-forms $(\sigma^g)^i_j$ of the metric $g$ express in terms of the
structure constants by
$$
(\sigma^g)^i_j(e_k) = -\frac12(g(e_i,[e_j,e_k]) - g(e_k,[e_i,e_j]) +
g(e_j,[e_k,e_i]))=\frac12(a^i_{jk}-a^k_{ij}+a^j_{ki}).
$$
Therefore, the connection 1-forms $(\sigma^c)^i_j$ for the Chern
connection $\nabla^c$ are determined by
\begin{equation}\label{chern-1-forms}
(\sigma^c)^i_j(e_k)=(\sigma^g)^i_j(e_k) - \frac12 C^i_j(e_k) =
\frac12(a^i_{jk}-a^k_{ij}+a^j_{ki}) - \frac12 C^i_j(e_k),
\end{equation}
where $C^i_j$ are the torsion 1-forms defined by
\begin{equation}\label{torsion-chern-1-forms}
C^i_j(e_k) =C(e_k,e_i,e_j)= dF(Je_k,e_i,e_j).
\end{equation}

\subsection{Adapted basis for balanced Hermitian
structures in Families I and II}\label{adapted-bases}

In this section we find an adapted basis for the balanced Hermitian
structures given in Families~I and~II of
Theorem~\ref{families-balanced-h19}.

To get an orthonormal coframe $\{e^1,\ldots,e^6\}$ with respect to a
metric in Family I it is enough to consider
\begin{equation}\label{es-family-I}
e^1+i\,e^2=r\,\omega^1,\quad e^3+i\,e^4=s\,\omega^2,\quad
e^5+i\,e^6=\omega^3.
\end{equation}
In terms of this basis, the structure equations become
\begin{equation}\label{str-eq-Family-I}
\begin{cases}
\begin{array}{lcl}
de^1 \zzz & = &\zzz de^2=de^5=0,\\
de^3 \zzz & = &\zzz \frac{2s}{r}\,e^{15},\\[2pt]  de^4 \zzz & = &\zzz
\frac{2s}{r}\,e^{25},\\[2pt] de^6 \zzz & = &\zzz
\pm\frac{2}{rs}\,(e^{13}+e^{24}),
\end{array}
\end{cases}
\end{equation} where the $(+)$-sign, resp. $(-)$-sign, in $de^6$ corresponds to
the complex structure $J_0^+$, resp. $J_0^-$.

The complex structures $J_0^+$ and $J_0^-$ are then given by $J_0^{\pm}
e^1=-e^2, J_0^{\pm} e^3=-e^4, J_0^{\pm} e^5=-e^6$, which means that the real
basis $\{e^1,\ldots,e^6\}$ is $J_0^{\pm}$-adapted. Therefore, the
fundamental form $F$ associated with the $J_0^{\pm}$-Hermitian metric
$g=e^1\otimes e^1+\cdots+ e^6\otimes e^6$ is given by
$F=e^{12}+e^{34}+e^{56}$. The structure
equations~\eqref{str-eq-Family-I} imply that
\begin{equation}\label{dF-torsion-h19-ii}
dF=\frac{-2s}{r}\,e^{145} + \frac{2s}{r}\,e^{235}\mp
\frac{2}{rs}(e^{135}+e^{245}).
\end{equation}

For Family II, we first consider the $(1,0)$-basis
$\{\sigma^j\}^3_{j=1}$ given by
$$
\sigma^1 = \omega^1,\quad\quad \sigma^2 = \omega^{2} +
\frac{i}{2\,p^2}\, \omega^{3} \, ,\quad\quad \sigma^3 =
i\,\omega^{2} - \frac{1}{2\,q^2}\, \omega^3,
$$
where $p^2$ and $q^2$ are the real, positive and different roots of
the polynomial $P(X)=t^2X^2 - s^2t^2X + \frac{s^2}{4}$. In
particular,
$$q^2=\frac{s^2t^2+\sqrt{s^2t^2\,(s^2t^2-1)}}{2\,t^2}>0,\quad p^2+q^2=s^2>0,
\quad p^2q^2=\frac{s^2}{4t^2}.$$ In terms of $\{\sigma^j\}^3_{j=1}$
the fundamental form expresses as
$2F=i\,(r^2\,\sigma^{1\bar1}+p^2\,\sigma^{2\bar2}+q^2\,\sigma^{3\bar3})$.
Now, the basis $\{e^1,\ldots,e^6\}$ given by
\begin{equation}\label{es-family-II}
e^1+i\,e^2=r\,\sigma^1,\quad\quad e^3+i\,e^4=p\,\sigma^2,\quad\quad
e^5+i\,e^6=q\,\sigma^3,
\end{equation}
is an adapted basis and the corresponding real structure equations
are:
\begin{equation}\label{str-eq-Family-II} \begin{cases}
\begin{array}{lcl}
de^1 \zzz & = &\zzz de^2=0,\\[4pt] de^3
\zzz & = &\zzz\frac{p}{r\,(p^2-q^2)}\,
\left[\mp\frac{1}{p}\,(e^{13}+e^{24})\pm\frac{q}{p^2}\,
(e^{16}-e^{25}) - 4pq\,(q\,e^{14}+p\,e^{15})\right],
\\[4pt] de^4 \zzz & = &\zzz\frac{-4 p^2 q}{r\,(p^2-q^2)}\,\left[q\,e^{24}+p\,e^{25}\right],
\\[4pt] de^5\zzz & = &\zzz\frac{4 p q^2}{r\,(p^2-q^2)}\,\left[q\,e^{24}+p\,e^{25}\right],
\\[4pt] de^6 \zzz & = &\zzz\frac{q}{r\,(p^2-q^2)}\,
\left[\mp\frac{p}{q^2}\,(e^{13}+e^{24})\pm\frac{1}{q}\,(e^{16}-e^{25})
-4 p q\,(q\,e^{14}+p\,e^{15})\right].
\end{array}
\end{cases}
\end{equation}
Taking into account \eqref{str-eq-Family-II}, $dF$ is expressed as:
\begin{equation}\label{dF-torsion-h19-i}\begin{array}{ll}
dF=\frac{1}{r\,(p^2-q^2)}&\Big[\mp\,(e^{134}-e^{156}) + 4 p q
(p^2+q^2)\,e^{145} \pm\frac{p}{q}\,(e^{135}+e^{245})\\
&\mp\frac{q}{p}\,(e^{146}-e^{245}) - 4 p^2 q^2\,(e^{234}-e^{256}) -4
p q\,(p^2\,e^{235} - q^2\,e^{246})\Big].\end{array}
\end{equation}

\section{Heterotic string compactifications based on non-nilpotent complex
structures}\label{soluciones}

\noindent In this section we find explicit solutions of the
heterotic supersymmetry equations with non-zero flux and constant
dilaton with respect to the Chern connection in the anomaly
cancellation condition based on invariant complex structures on
nilmanifolds. From the considerations in
Section~\ref{het-constant-dil} we are led to non-nilpotent complex
structures admitting balanced compatible metric, i.e. to a compact
nilmanifold $N$ corresponding to the Lie algebra
$\frn=(0,0,0,12,23,14-35)$. In Section~\ref{explicit} below we
provide an explicit realization of $N$ as a compact quotient
$\Gamma\backslash K$ of a simply connected nilpotent Lie group $K$
by a lattice $\Gamma$ of maximal rank.

In Section~\ref{complex-sect} we have proved that any invariant
balanced Hermitian structure on $N$ is equivalent to $(J_0^{\pm},F)$,
where $F$ belongs to the Family I or II of
Theorem~\ref{families-balanced-h19}. Next we look for solutions of
the Strominger system in each family separately.

\subsection{Solutions in Family I}\label{family-I}

We consider first the Family I of balanced Hermitian
SU(3)-structures $(J_0^{\pm},F,\Psi)$ on~$N$. Since $\{e^i\}^6_{i=1}$
given in \eqref{es-family-I} is an adapted basis, from
\eqref{dF-torsion-h19-ii} it follows that the torsion $T$ is given
by
$$T=J_0^{\pm} dF=\frac{2s}{r}\,e^{146} -
\frac{2s}{r}\,e^{236}\pm \frac{2}{rs}(e^{136}+e^{246}),
$$
which implies, using the structure equations
\eqref{str-eq-Family-I}, that
$$dT=-\frac{8}{r^2}\left(\frac{1}{s^2}\,e^{1234}+s^2\,e^{1256}\right).$$

Next we find a large family of SU(3)-instantons for any structure
$(J_0^{\pm},F,\Psi)$.

\begin{prop}\label{instantons}
For each $\lambda,\mu,\tau \in \mathbb{R}$, let
$A_{\lambda,\mu,\tau}$ be the {\rm SU(3)}-connection defined by the
connection $1$-forms
$$(\sigma^{A_{\lambda,\mu,\tau}})^2_3=(\sigma^{A_{\lambda,\mu,\tau}})^2_5=
(\sigma^{A_{\lambda,\mu,\tau}})^4_5=\frac12(\sigma^{A_{\lambda,\mu,\tau}})^5_6=-\lambda\,e^1
-\mu\,e^2 - \tau\,e^6,\qquad
(\sigma^{A_{\lambda,\mu,\tau}})^i_j=\lambda\,e^1+ \mu\,e^2 +
\tau\,e^6,$$ for $1\leq i< j\leq 6$ such that $(i,j)\neq (2,3),
(2,5), (4,5), (5,6)$, and $\sigma^j_i=-\sigma^i_j$. Then,
$A_{\lambda,\mu,\tau}$ is an {\rm SU(3)}-instanton and
$$p_1(A_{\lambda,\mu,\tau})=\frac{-18\,\tau^2}{\pi^2 r^2s^2}\,e^{1234}.$$
\end{prop}

\begin{proof}
Since $\{e^1,\ldots,e^6\}$ is an adapted basis for the
SU(3)-structure and the connection 1-forms with respect to this
basis satisfy $\sigma^j_i=-\sigma^i_j$ and
$$\sigma^1_3=\sigma^2_4,\quad \sigma^1_4=-\sigma^2_3,\quad
\sigma^1_5=\sigma^2_6,\quad
\sigma^1_6=-\sigma^2_5,\quad\sigma^3_5=\sigma^4_6,\quad
\sigma^3_6=-\sigma^4_5, \quad \sigma^1_2+\sigma^3_4+\sigma^5_6=0,$$
the connection $A_{\lambda,\mu,\tau}$ preserves $F$ and $\Psi$, i.e.
it is an SU(3)-connection. A direct calculation, using
\eqref{curvature} and \eqref{str-eq-Family-I}, shows that the
curvature forms $(\Omega^{A_{\lambda,\mu,\tau}})^i_j$ of the
connection $A_{\lambda,\mu,\tau}$ are given by
$$(\Omega^{A_{\lambda,\mu,\tau}})^2_3=(\Omega^{A_{\lambda,\mu,\tau}})^2_5=
(\Omega^{A_{\lambda,\mu,\tau}})^4_5=\frac12(\Omega^{A_{\lambda,\mu,\tau}})^5_6=\frac{-2\tau}{rs}\,(e^{13}+e^{24}),\qquad
(\Omega^{A_{\lambda,\mu,\tau}})^i_j=\frac{2\tau}{rs}\,(e^{13}+e^{24}).
$$
Since the 2-forms $(\Omega^{A_{\lambda,\mu,\tau}})^i_j$ satisfy
equations~\eqref{inst}, the connection $A_{\lambda,\mu,\tau}$ is an
SU(3)-instanton.
\end{proof}

In order to obtain solutions to the anomaly cancellation condition
with respect to the Chern connection, we compute first the curvature
of $\nabla^c$. According to \eqref{torsion-chern-1-forms} and
\eqref{dF-torsion-h19-ii} the torsion 1-forms $C^i_j$ of the Chern
connection are $C^1_2=C^1_6=C^2_6=C^3_4=C^3_6=C^4_6=C^5_6=0$ and
$$\begin{array}{lll}
C^1_3=C^2_4=\mp\frac{2}{rs} e^6,\qquad &
C^1_4=-C^2_3=-\frac{2s}{r} e^6,\qquad & C^1_5=-\frac{2s}{r} e^3 \pm\frac{2}{rs} e^4,\\[4pt]
C^2_5=\mp\frac{2}{rs} e^3 -\frac{2s}{r} e^4,\qquad &
C^3_5=-\frac{2s}{r} e^1 \mp\frac{2}{rs} e^2,\qquad &
C^4_5=\pm\frac{2}{rs} e^1 -\frac{2s}{r} e^2.
\end{array}$$
From \eqref{chern-1-forms} and the structure equations
\eqref{str-eq-Family-I} it follows that the non-zero Chern
connection 1-forms $(\sigma^c)^i_j$ are the following:
$$
\begin{array}{lll}
(\sigma^c)^1_3=(\sigma^c)^2_4=-\frac{s}{r}\,e^5,\quad &
(\sigma^c)^1_4=-(\sigma^c)^2_3=\frac{s}{r}\,e^6,\quad&
(\sigma^c)^1_5=(\sigma^c)^2_6=\mp\frac{1}{rs}\,e^4,\\[5pt]
(\sigma^c)^1_6=-(\sigma^c)^2_5=\mp\frac{1}{rs}\,e^3,
&(\sigma^c)^3_5=(\sigma^c)^4_6=\pm\frac{1}{rs}\,e^2,
&(\sigma^c)^3_6=-(\sigma^c)^4_5=\pm\frac{1}{rs}\,e^1.
\end{array}
$$

Now, using \eqref{curvature} we get the curvature 2-forms
$(\Omega^c)^i_j$ for the Chern connection $\nabla^c$:
\begin{eqnarray*}
(\Omega^c)^1_2&=& -\frac{2}{r^2s^2}(e^{34} +
s^4\,e^{56}),\\
(\Omega^c)^1_3&=& (\Omega^c)^2_4 = -\frac{1}{r^2s^2}(e^{13}+e^{24}),\\
(\Omega^c)^1_4&=& -(\Omega^c)^2_3 = \pm\frac{2}{r^2}(e^{13}+e^{24})
+
\frac{1}{r^2s^2}(e^{14}-e^{23}),\\
(\Omega^c)^1_5&=& (\Omega^c)^2_6 =\pm\frac{1}{r^2}(e^{16}-e^{25}),\\
(\Omega^c)^1_6&=& -(\Omega^c)^2_5 = \mp\frac{1}{r^2}(e^{15}+e^{26}),\\
(\Omega^c)^3_4&=& -\frac{2}{r^2s^2}(e^{12} -
s^4\,e^{56}),\\
(\Omega^c)^3_5&=& (\Omega^c)^4_6 =\mp\frac{1}{r^2}(e^{36}-e^{45}),\\
(\Omega^c)^3_6&=& -(\Omega^c)^4_5 =\pm\frac{1}{r^2}(e^{35}+e^{46}),\\
(\Omega^c)^5_6&=& -(\Omega^c)^1_2-(\Omega^c)^3_4.
\end{eqnarray*}

Hence, the first Pontrjagin class is represented by
$$
p_1(\nabla^c) =-\frac{2}{\pi^2 r^4} (e^{1234}+e^{1256}).
$$

Therefore, we have proved the following result for the structures
belonging to Family I whenever the metric coefficient $s^2\geq 1$:

\begin{thrm}\label{solutions-Family-I}
Let $A_{\lambda,\mu,\tau}$ be the {\rm SU(3)}-instanton given in
{\rm Proposition~\ref{instantons}} with
$\tau^2=\frac{s^4-1}{9r^2s^2}$. Then,
$$dT= 2\pi^2 \alpha'\,\left(p_1(\nabla^c)-p_1(A_{\lambda,\mu,\tau})\right),\quad\quad \alpha'=2r^2s^2>0.$$
Therefore, $(N, J_0^{\pm}, F, \Psi, A_{\lambda,\mu,\tau}, \nabla^c)$ is
a compact solution of the heterotic supersymmetry equations with
non-zero flux, non-flat instanton and constant dilaton with respect
to the Chern connection in the anomaly cancellation condition, for
$r\not=0$, $s^2\geq 1$, $\tau^2=\frac{s^4-1}{9r^2s^2}$ and for any
$\lambda,\mu$.
\end{thrm}

\begin{rmrk}\label{deformation}{\rm
For $J_0^+$ and $r^2=s^2=1$ we get the particular solution given in
\cite{FIUV}. Notice that for that solution the instanton is flat, so
the solutions given in Theorem~\ref{solutions-Family-I} can be
thought as a deformation of that particular solution but with
instanton having curvature with non-zero trace.}
\end{rmrk}

\subsection{Solutions in Family II}\label{family-II}

Now we consider the Family II of balanced Hermitian SU(3)-structures
$(J_0^{\pm},F,\Psi)$ on $N$. The basis $\{e^i\}^6_{i=1}$ given in
\eqref{es-family-II} is adapted to the structure and from
(\ref{dF-torsion-h19-i}) the torsion~$T$~is
\begin{equation}\label{torsion-h19-i}\begin{array}{ll}
T=J_0^{\pm} dF=\frac{1}{r\,(p^2-q^2)} \Big[\!\!\!&\!\!\!
\pm\,(e^{234}-e^{256}) + 4 p
q(p^2+q^2)\,e^{236}\mp\frac{p}{q}\,(e^{136}+e^{246})\\[6pt]
\!\!\!&\!\!\! \pm\frac{q}{p}\,(e^{235}-e^{136}) - 4 p^2
q^2\,(e^{134}-e^{156}) - 4 p q\,(p^2\,e^{146} -
q^2\,e^{135})\Big].\end{array}
\end{equation}
From the structure equations \eqref{str-eq-Family-II} we get
$$dT=\frac{-2(p^2+q^2)}{r^2p^2q^2(p^2-q^2)^2} \Big[ (1+16p^2q^6)p^2 e^{1234}
+(1+16p^4q^4)pq (e^{1235}+e^{1246})+(1+16p^6q^2)q^2 e^{1256}
\Big].$$

From \eqref{torsion-chern-1-forms} and \eqref{dF-torsion-h19-i} it
follows that the non-zero torsion 1-forms $C^i_j$ of the Chern
connection $\nabla^c$ are
$$
\begin{array}{lll}
& C^1_3= \frac{1}{r(p^2-q^2)}(\pm\,e^3 \pm \frac{p}{q} e^6),\quad
& C^2_6= \frac{1}{r(p^2-q^2)}(4p q^3 e^3 - 4 p^2q^2 e^6), \\[6pt]
& C^1_4= \frac{1}{r(p^2-q^2)}(\pm\,e^4 \pm\frac{q}{p} e^5 +4p
q(p^2+q^2)e^6),\quad
& C^3_4= \frac{1}{r(p^2-q^2)}(4p^2 q^2 e^1 \mp e^2), \\[6pt]
& C^1_5= \frac{1}{r(p^2-q^2)}(4pq(p^2+q^2) e^3 \mp \frac{p}{q} e^4
\mp\, e^5),\quad
& C^3_5= \frac{1}{r(p^2-q^2)}(4p^3 q\, e^1 \pm\frac{p}{q} e^2), \\[6pt]
& C^1_6= \frac{1}{r(p^2-q^2)}(\mp \frac{q}{p} e^3 \mp\, e^6),\quad
&C^4_5= \frac{1}{r(p^2-q^2)}(\mp\frac{p^2+q^2}{pq} e^1 +4pq(p^2+q^2) e^2), \\[6pt]
& C^2_3= \frac{1}{r(p^2-q^2)}(4p^2 q^2 e^3 - 4p^3 q\, e^6),\quad
& C^4_6= \frac{1}{r(p^2-q^2)}(-4p q^3 e^1 \mp\frac{q}{p} e^2), \\[6pt]
& C^2_4= \frac{1}{r(p^2-q^2)}(4p^2 q^2 e^4 - 4pq^3 e^5
\pm\frac{p^2+q^2}{pq} e^6),\quad
& C^5_6= \frac{1}{r(p^2-q^2)}(-4p^2 q^2 e^1 \pm e^2). \\[6pt]
& C^2_5= \frac{1}{r(p^2-q^2)}(\pm \frac{p^2+q^2}{pq} e^3 + 4 p^3 q\,
e^4 -4p^2q^2 e^5), &
\end{array}
$$
and therefore, using \eqref{chern-1-forms} and the structure
equations \eqref{str-eq-Family-II}, the non-zero Chern connection
1-forms $(\sigma^c)^i_j$ are the following:
$$
\begin{array}{ll}
& (\sigma^c)^1_3=(\sigma^c)^2_4= \frac{1}{2r(p^2-q^2)}(\pm\,e^3+4
p^2q^2 e^4+4p^3 q\, e^5\mp\frac{q}{p}
e^6),\\[6pt]
& (\sigma^c)^1_4=-(\sigma^c)^2_3=\frac{1}{2r(p^2-q^2)}(4p^2q^2 e^3
\mp\, e^4 \mp \frac{q}{p} e^5 -4p^3 q\, e^6),\\[6pt]
& (\sigma^c)^1_5=(\sigma^c)^2_6= \frac{1}{2r(p^2-q^2)}(-4pq^3 e^3 \pm\frac{p}{q} e^4 \pm\, e^5 +4p^2 q^2 e^6),\\[6pt]
&
(\sigma^c)^1_6=-(\sigma^c)^2_5=\frac{1}{2r(p^2-q^2)}(\pm\frac{p}{q}
e^3 + 4pq^3 e^4 +4p^2q^2 e^5 \mp\, e^6),\\[6pt]
&(\sigma^c)^3_4=-(\sigma^c)^5_6=\pm\frac{1}{r(p^2-q^2)} e^2,\\[6pt]
&(\sigma^c)^3_5=(\sigma^c)^4_6=\mp\frac{1}{2rp q} e^2,\\[6pt]
&(\sigma^c)^3_6=-(\sigma^c)^4_5=\mp\frac{p^2+q^2}{2rpq(p^2-q^2)}
e^1.
\end{array}
$$

By \eqref{curvature} we get the curvature 2-forms $(\Omega^c)^i_j$
for the Chern connection $\nabla^c$:
\begin{eqnarray*}
(\Omega^c)^1_2  \!\!\!&=&\!\!\! \frac{-1}{r^2(p^2-q^2)^2} \, \Bigl[
\frac{(p^2+q^2)(1+16p^2q^6)}{2q^2}e^{34}+\frac{(p^2+q^2)(1+16p^6q^2)}{2p^2}e^{56}\\
& & +\frac{(p^2+q^2)(1+16p^4q^4)}{2pq}(e^{35}+e^{46}) \pm\,
4pq(p^2-q^2)(e^{36}-e^{45})
\Bigr],\\[6pt]
(\Omega^c)^1_3 \!\!\!&=&\!\!\! (\Omega^c)^2_4 =
\frac{-1}{r^2(p^2-q^2)^2} \, \Bigl[
\frac{p^2+q^2}{4q^2}(e^{13}+e^{24}) \pm\, q^2(3p^2-q^2)(e^{14}-e^{23})\\
& & \quad\quad\quad\pm\, p q(3p^2-q^2) (e^{15}+e^{26}) -
\frac{p^2+q^2}{4p q}(e^{16}-e^{25}) \Bigr],\\[6pt]
(\Omega^c)^1_4 \!\!\!&=&\!\!\! -(\Omega^c)^2_3 =
\frac{1}{r^2(p^2-q^2)^2} \, \Bigl[ \pm
(2p^4-3p^2q^2-q^4)(e^{13}+e^{24}) +
\frac{p^2+q^2}{4q^2}(e^{14}-e^{23})\\
& & \quad\quad\quad+ \frac{p^2+q^2}{4p q} (e^{15}+e^{26}) \mp\, p
q(p^2-3q^2)(e^{16}-e^{25}) \Bigr],\\[6pt]
(\Omega^c)^1_5 \!\!\!&=&\!\!\! (\Omega^c)^2_6 =
\frac{1}{r^2(p^2-q^2)^2} \, \Bigl[ \mp
p q(3p^2-q^2)(e^{13}+e^{24}) + \frac{p^2+q^2}{4 p q}(e^{14}-e^{23})\\
& & \quad\quad\quad+ \frac{p^2+q^2}{4 p^2} (e^{15}+e^{26}) \pm
(p^4+3p^2q^2-2q^4)(e^{16}-e^{25}) \Bigr],\\[6pt]
(\Omega^c)^1_6 \!\!\!&=&\!\!\! -(\Omega^c)^2_5 =
\frac{1}{r^2(p^2-q^2)^2} \, \Bigl[ \frac{p^2+q^2}{4p
q}(e^{13}+e^{24})
\mp\, p q(p^2-3q^2) (e^{14}-e^{23})\\
& & \quad\quad\quad\mp\, p^2(p^2-3q^2) (e^{15}+e^{26}) -
\frac{p^2+q^2}{4p^2}(e^{16}-e^{25}) \Bigr],
\end{eqnarray*}

\begin{eqnarray*}
(\Omega^c)^3_4 \!\!\!&=&\!\!\! \frac{-1}{r^2(p^2-q^2)^2} \, \Bigl[
\frac{p^4-q^4}{2p^2 q^2} e^{12}- \frac{1+16p^4q^4}{2}e^{34} -
\frac{q^2(1+16p^8)}{2p^2}e^{56}\\
& & -\frac{q(1+16p^6q^2)}{2p}(e^{35}+e^{46})
\mp 2p q(p^2-q^2)(e^{36}-e^{45})\Bigr],\\[6pt]
(\Omega^c)^3_5 \!\!\!&=&\!\!\! (\Omega^c)^4_6 =
\frac{-1}{r^2(p^2-q^2)^2} \, \Bigl[ \frac{p^2+q^2}{p q} e^{12}+
\frac{p(1+16p^2q^6)}{2q}e^{34} +
\frac{q(1+16p^6q^2)}{2p}e^{56}\\
& & \quad\quad\quad+\frac{1+16p^4q^4}{2}(e^{35}+e^{46})
\pm(p^4-q^4)(e^{36}-e^{45})\Bigr],\\[6pt]
(\Omega^c)^3_6 \!\!\!&=&\!\!\! -(\Omega^c)^4_5 = \frac{\pm
1}{r^2(p^2-q^2)} \, \Bigl[
2p q (e^{34}+e^{56}) + (p^2+q^2)(e^{35}+e^{46}) \Bigr],\\[6pt]
(\Omega^c)^5_6 \!\!\!&=&\!\!\! -(\Omega^c)^1_2-(\Omega^c)^3_4.
\end{eqnarray*}

Hence, the first Pontrjagin class of the Chern connection is
represented by
$$
p_1(\nabla^c) =-\frac{2(p^2+q^2)}{\pi^2
r^4(p^2-q^2)^2}\left((p^2+q^2) (e^{1234}+e^{1256}) +2p q
(e^{1235}+e^{1246}) \right).
$$

The following result gives explicit solutions with respect to the
Chern connection for the structures belonging to Family II whenever
the metric coefficients $s^2$ and $t^2$ are equal and greater
than~1.

\begin{thrm}\label{solutions-Family-II}
The connection $A_{\lambda,\mu,\tau}$ given in {\rm
Proposition~\ref{instantons}} is an {\rm SU(3)}-instanton for the
structures given in {\rm Family II} if and only if $\tau=0$. With
respect to such an instanton, there are solutions to the anomaly
cancellation condition if and only if $s^2=t^2$, in which case
$$dT=2\pi^2\alpha'\,\left(p_1(\nabla^c)-p_1(A_{\lambda,\mu,0})\right)=2\pi^2\alpha'\,p_1(\nabla^c),\quad\quad \alpha'=2r^2>0.$$
Therefore, $(N, J_0^{\pm}, F, \Psi, A_{\lambda,\mu,0}, \nabla^c)$ is a
compact solution of the heterotic supersymmetry equations with
non-zero flux and constant dilaton with respect to the Chern
connection in the anomaly cancellation condition, for $r\not= 0$,
$s^2=t^2> 1$ and for any $\lambda,\mu$.
\end{thrm}

\begin{proof}
A direct calculation shows that the connection
$A_{\lambda,\mu,\tau}$ given in Proposition~\ref{instantons} is
again an SU(3)-connection, but it is an SU(3)-instanton for a
structure in Family II if and only if $\tau=0$, what implies that
the curvature vanishes identically. Now, comparing the coefficients
of $e^{1235}$ and $e^{1246}$ in $dT$ and $p_1(\nabla^c)$, we get
that if there is $\alpha'$ such that $dT=\alpha'\,p_1(\nabla^c)$
then necessarily $\alpha'=\frac{r^2(1+16p^4q^4)\pi^2}{2p^2q^2}$. But
in this case comparing the coefficients of $e^{1234}$ and $e^{1256}$
we get that $dT$ is a positive multiple of $p_1(\nabla^c)$ if and
only if $p^2-q^2=16p^4q^4(p^2-q^2)$. Since $p^2\not= q^2$ we
conclude that $p^2q^2=\frac{1}{4}$. Notice that this is equivalent
to the metric coefficients $s^2$ and $t^2$ to be equal, because
$p^2q^2=\frac{s^2}{4t^2}$.
\end{proof}

As a consequence of Propositions~\ref{str-eq-red-h19}
and~\ref{h19-h26}, and Theorems~\ref{families-balanced-h19},
\ref{solutions-Family-I} and~\ref{solutions-Family-II}, we conclude
the existence of many solutions to the Strominger system for any
invariant complex structure $J$ on $N$. In fact, such a $J$ is
isomorphic to $J_0^+$ or $J_0^-$, and for each one of these two
complex structures there is a four-parametric family of
(non-equivalent) balanced Hermitian structures providing solutions:
a two-parametric family corresponding to structures in Family I and
another two-parametric family corresponding to structures in Family
II.

\begin{cor}\label{generalresult}
Let $J$ be any invariant complex structure on $N$. Then, there are
{\rm SU(3)}-structures on the compact complex 6-manifold $(N,J)$
solving the heterotic supersymmetry equations with non-zero flux and
constant dilaton with respect to the Chern connection in the anomaly
cancellation condition.
\end{cor}

\begin{rmrk}\label{eq-motion}
{\rm It has been proved recently \cite{I} that the heterotic
supersymmetry and the anomaly cancellation imply the heterotic
equations of motion if and only if the connection on the tangent
bundle in the $\alpha'$ corrections is an SU(3)-instanton. From the
explicit expressions of the curvature forms $(\Omega^c)^i_j$ above,
one can see that the Chern connection is never an SU(3)-instanton.}
\end{rmrk}

\subsection{Explicit realizations}\label{explicit}

Let us consider the 6-dimensional Lie group
\begin{equation*}
K=\left\{\begin{pmatrix}1&\frac12 x_3&x_5&x_4&x_6\\
 0&1&2x_2&2x_1&x_1^2+x_2^2\\
 0&0&1&0&x_2\\
 0&0&0&1&x_1\\
 0&0&0&0&1\end{pmatrix}\,|\, x_i\in\mathbb{R}\right\}.
 \end{equation*}
The left translation $L$ by an element $(a_1,\ldots,a_6)\in K$ is
given by
$$
L(x_1,\ldots,x_6)= ( a_1+x_1, a_2+x_2, a_3+x_3, a_4+a_3x_1+x_4,
a_5+a_3x_2+x_5, a_6+a_5x_2+a_4x_1+\frac{1}{2}a_3(x_1^2+x_2^2)+x_6 ).
$$

The following vector fields constitute a basis of left-invariant
vector fields on the Lie group $K$:
$$
X_1=\frac{\partial}{\partial x_1} + x_3 \frac{\partial}{\partial
x_4} + x_4 \frac{\partial}{\partial x_6},\quad\quad
X_2=\frac{\partial}{\partial x_2} + x_3\frac{\partial}{\partial x_5}
+ x_5 \frac{\partial}{\partial x_6},\quad\quad
X_i=\frac{\partial}{\partial x_i}, \quad i=3,4,5,6.
$$
Therefore, the Lie algebra of $K$ is generated by elements
$X_1\ldots,X_6$ with non-zero brackets given by
$$[X_1,X_3]=-X_4,\quad [X_2,X_3]=-X_5,\quad
[X_1,X_4]=[X_2,X_5]=-X_6.$$

This Lie algebra is easily seen to be isomorphic to
$\frn=(0,0,0,12,23,14-35)$. In fact, with respect to the dual basis
of $\{X_1,\ldots,X_6\}$, which is given by
$$
\beta^1=dx_1,\ \beta^2=dx_2,\ \beta^3=dx_3,\ \beta^4=dx_4-x_3dx_1,\
\beta^5=dx_5-x_3dx_2,\ \beta^6=dx_6-x_4dx_1-x_5dx_2,$$ we have that
the structure equations of $K$ become
$d\beta^1=d\beta^2=d\beta^3=0$, $d\beta^4=\beta^{13}$,
$d\beta^5=\beta^{23}$, $d\beta^6=\beta^{14}+\beta^{25}$, and thus
the relations
$$e^1=\frac{r}{2s}\beta^1,\quad e^2=\frac{r}{2s}\beta^2,\quad e^3=\beta^4,\quad
e^4=\beta^5,\quad e^5=\beta^3,\quad e^6=\pm\frac{1}{s^2}\beta^6,$$
provide the explicit isomorphism between $\frn$ and the Lie algebra
of $K$.

Now, in terms of the coordinates $(x_1,\ldots,x_6)$ the balanced
Hermitian metrics of the Family I are:
\begin{eqnarray*}
g_{r,s} \!\!\!&=&\!\!\! \frac{r^2}{4s^2} ((dx_1)^2 + (dx_2)^2) +
(dx_3)^2 + (dx_4-x_3dx_1)^2 + (dx_5-x_3dx_2)^2 +
[\pm\frac{1}{s^2}(dx_6-x_4dx_1-x_5dx_2)]^2\\[7pt]
\!\!\!&=&\!\!\! \left( \frac{r^2}{4s^2} +(x_3)^2 +\frac{1}{s^4}
(x_4)^2 \right) (dx_1)^2 + \frac{2}{s^4}x_4x_5\, dx_1dx_2 - 2 x_3\,
dx_1dx_4 - \frac{2}{s^4}x_4\, dx_1dx_6\\[5pt]
&\!\!\!&\!\!\! + \left( \frac{r^2}{4s^2} +(x_3)^2 +\frac{1}{s^4}
(x_5)^2 \right)
(dx_2)^2 - 2 x_3\, dx_2dx_5 - \frac{2}{s^4}x_5\, dx_2dx_6\\[5pt]
&\!\!\!&\!\!\! +\, (dx_3)^2 + (dx_4)^2 + (dx_5)^2 + \frac{1}{s^4}
(dx_6)^2.
\end{eqnarray*}

We can express in a similar way the complex structures $J_0^{\pm}$, the
form $\Psi$ and the instantons $A_{\lambda,\mu,\tau}$ in terms of
the coordinates $(x_1,\ldots,x_6)$. Finally, notice that the
discrete subgroup $\Gamma$ can be taken as the subgroup of $K$
consisting of those matrices with integers entries, so that
$N=\Gamma\backslash K$.

\medskip
\noindent {\bf Acknowledgments.}
This work has been partially supported through Projects MICINN
(Spain) MTM2008-06540-C02-02 and MTM2011-28326-C02-01.

\end{document}